\documentclass[a4paper,11pt]{article}
\usepackage[T1]{fontenc}
\usepackage[latin1]{inputenc}
\usepackage{amssymb}
\usepackage[english]{babel}
\usepackage{amsthm}
\usepackage{amsmath}
\usepackage{amsfonts}
\usepackage{enumerate}

\usepackage{array,lmodern,graphicx}

\def\R{\mathbb{R}}

\def\fra{\mathfrak{a}}

\newtheorem{lem}{Lemma}[section]

\newtheorem{prop}[lem]{Proposition}
\newtheorem{teo}[lem]{Theorem}

\def\G{\mathcal{G}}
\def\H{\mathcal{H}}
\def\eps{\varepsilon}

\begin{document}

\title{Littlewood-Paley-Stein functions for Schr\"odinger operators}
\author{El Maati Ouhabaz \footnote{This note is published in "Frontiers in Sciences and Engineerings", Edited by the Hassan II Academy of Science and Technology of Morocco, Vol. 6, no 1 (2016) 99-109.} \thanks{\noindent Univ. Bordeaux, Institut de Math\'ematiques (IMB). CNRS UMR 5251. 351,  
Cours de la Lib\'eration 33405 Talence, France.
 Elmaati.Ouhabaz@math.u-bordeaux.fr.}}
% \footnote{Paper appeared in "Frontiers in Sciences and Engineerings", Edited by the Hassan II Academy of Science and Technology of Morocco, Vol. 6, no 1 (2016) 99-109}
  
\date{}

\maketitle

\vspace{1cm} 
\begin{center}{{\it Dedicated to the memory of Abdelghani Bellouquid\\
(2/2/1966--8/31/2015)}}
\end{center}
\vspace{1cm}

\begin{abstract}
We study the boundedness on $L^p(\R^d)$ of the vertical Littlewood-Paley-Stein functions for Schr\"odinger operators $-\Delta + V$ with non-negative potentials
$V$. These functions are proved to be bounded on $L^p$ for all $p \in (1,2)$. The situation for $p > 2$ is different. We prove  for a class of potentials that the boundedness on $L^p$,  for some $p > d$, holds {\it if and only if} $V = 0$.
\end{abstract}
\vspace{.5cm}

\vspace{.5cm}

\section{Introduction}\label{sec1}

Let $L := -\Delta + V$ be a Schr\"odinger operator with a non-negative potential $V$.  It  is the  self-adjoint operator  associated with the  form
$$\fra(u,v) := \int_{\R^d} \nabla u. \nabla v dx + \int_{\R^d} V u v dx$$
with domain
$$ D(\fra) = \{ u \in W^{1,2}(\R^d), \int_{\R^d} V |u|^2 dx < \infty \}.$$
We denote by $(e^{-tL})_{t\ge0}$ the semigroup generated by (minus) $L$ on $L^2(\R^d)$. Since $V$ is nonnegative, it follows from the Trotter product formula that 
\begin{equation}\label{01}
0 \le e^{-tL} f \le e^{t\Delta}f 
\end{equation}
for all $t \ge 0$ and  $0 \le f \in L^2(\R^d)$ (all the inequalities are in the a.e. sense). It follows immediately from \eqref{01} that the semigroup $(e^{-tL})_{t\ge0}$  is sub-Markovian and hence extends to a contraction $C_0$-semigroup on $L^p(\R^d)$ for all 
$p \in [1, \infty)$.  We shall also denote by $(e^{-tL})_{t\ge0}$  the corresponding semigroup on $L^p(\R^d)$. 

The domination property \eqref{01} implies in particular that the corresponding heat kernel of $L$ is pointwise bounded by the Gaussian heat kernel.   As a consequence, $L$ has a bounded holomorphic functional calculus on $L^p(\R^d)$ and even H\"ormander type functional calculus (see \cite{DOS}). This implies the boundedness on $L^p(\R^d)$ for all $p \in (1, \infty)$ of the horizontal  Littlewood-Paley-Stein functions:
$$ g_L(f)(x) := \left[ \int_0^\infty t | \sqrt{L} e^{-t\sqrt{L}} f(x) |^2 dt \right]^{1/2}$$
and 
$$ h_L(f)(x) := \left[ \int_0^\infty t | L e^{-t L} f(x) |^2 dt \right]^{1/2}.$$
Indeed, these functions are of the form  (up to a constant)
$$  S_L f(x) = \left[ \int_0^\infty |  \psi(tL) f(x) |^2 \frac{dt}{t} \right]^{1/2}$$
with $\psi(z) = \sqrt{z} e^{- \sqrt{z}}$ for $g_L$ and $\psi(z) = z e^{-z}$ for $h_L$. The boundedness of the holomorphic functional calculus 
implies the boundedness of $S_L$ (see \cite{CDMY}). Thus,  $g_L$ and $h_L$ are bounded on $L^p(\R^d)$ for all $p \in (1, \infty)$ and this holds for every nonnegative potential $V \in L_{loc}^1(\R^d)$. 

Now we define the so-called  vertical Littlewood-Paley-Stein functions
$$ \G_L(f)(x) := \left( \int_0^\infty t | \nabla e^{-t\sqrt{L}} f(x) |^2 +  t | \sqrt{V}  e^{-t\sqrt{L}} f(x) |^2 \ dt \right)^{1/2} $$
and 
$$\H_L(f)(x) := \left( \int_0^\infty | \nabla e^{-t L} f(x) |^2 +   | \sqrt{V}  e^{-t L} f(x) |^2 \ dt \right)^{1/2}. $$

Note that usually, these two functions are defined without the additional terms $t | \sqrt{V}  e^{-t\sqrt{L}} f(x) |^2$ and $ | \sqrt{V}  e^{-t L} f(x) |^2$. 

The functions $\G_L$ and $\H_L$ are very different from $g_L$ and $h_L$ as we shall see in the last section of this paper. 
If $V = 0$ and hence $L = -\Delta$ it is a very well known fact that $\G_L$ and $\H_L$ are bounded on $L^p(\R^d)$ for all 
$p \in (1, \infty)$. The Littlewood-Paley-Stein  functions are crucial in the study 
of non-tangential limits  of Fatou type and the boundedness of Riesz
transforms. We refer to  \cite{Stein1}-\cite{Stein2}.
 For Schr\"odinger operators, boundedness results on $L^p(\R^d)$ are proved in  \cite{Shigekawa} for 
potentials $V$ which satisfy  $\frac{| \nabla V |}{V} + \frac{\Delta V}{V} \in L^\infty(\R^d)$. This is a rather restrictive condition. 
For elliptic operators in divergence form (without a potential)  boundedness results on $L^p(\R^d)$ for certain values of $p$ are proved in \cite{Auscher}. For the setting of Riemannian manifolds we refer to \cite{Coulhon-Studia} and \cite{CoulhonCPAM}. Again the last two papers do not deal with Schr\"odinger operators.  

In this note we prove that $\G_L$ and $\H_L$ are bounded on $L^p(\R^d)$ for all $p \in (1, 2]$ for every nonnegative potential $V \in L_{loc}^1(\R^d)$.  That is 
$$\int_{\R^d} \left( \int_0^\infty t | \nabla e^{-t\sqrt{L}} f(x) |^2 +  t | \sqrt{V}  e^{-t\sqrt{L}} f(x) |^2 \ dt \right)^{p/2} dx \le C \int_{\R^d} |f(x)|^p dx$$
and similarly,
$$\int_{\R^d} \left( \int_0^\infty  | \nabla e^{-t L} f(x) |^2 +  | \sqrt{V}  e^{-t L} f(x) |^2 \ dt \right)^{p/2} dx \le C \int_{\R^d} |f(x)|^p dx$$
for all $f \in L^p(\R^d)$. \\
Our  arguments of the proof are borrowed from the paper \cite{Coulhon-Studia} which we adapt to our case in order to take into account the terms with $\sqrt{V}$ in the definitions of $\G_L$ and $\H_L$. Second we consider the case $p > 2$ and $d \ge 3$.  
For a  wide class of potentials, we prove that if $\G_L$ (or $\H_L$) is bounded on $L^p(\R^d)$ for some 
$p > d$ then $V= 0$. Here we use some ideas from \cite{Magniez} which  deals with the Riesz transform on Riemannian manifolds.  In this latter result we could replace $\G_L$ by $\left( \int_0^\infty t | \nabla e^{-t\sqrt{L}} f(x) |^2 dt \right)^{1/2}$ and the conclusion remains valid. 

Many questions of harmonic analysis  have been studied for  Schr\"odinger operators. For example, spectral multipliers and Bochner Riesz means \cite{DOS} and \cite{Ouh05} and  Riesz transforms \cite{Ouh05}, \cite{Assaad}, \cite{Shen} and \cite{AusBA}. 
However little seems to be available in the literature concerning  the  associated Littlewood-Paley-Stein functions $\G_L$ and $\H_L$. 
Another reason which motivates the present paper is to understand the Littlewood-Paley-Stein functions for the Hodge de-Rham Laplacian on differential forms. Indeed, Bochner's formula allows to write the Hodge de-Rham Laplacian on $1$-differential forms as a Schr\"odinger operator (with a vector-valued potential). Hence, understanding the Littlewood-Paley-Stein functions for Schr\"odinger operators $L$  could be  a first step in order to consider the Hodge de-Rham Laplacian. Note however that unlike the present case, if the manifold has a negative Ricci curvature part, then the semigroup of the Hodge de-Rham Laplacian does not necessarily act on all $L^p$ spaces. Hence the arguments presented in this paper have to be changed considerably. We shall address  this problem in a forthcoming  paper. 

\section{Boundedness on $L^p, \ 1 < p \le 2$}\label{sec2}
Recall that $L = -\Delta + V$ on $L^2(\R^d)$. We have 
\begin{teo}\label{thm1} 
For every  $ 0 \le V \in L_{loc}^1(\R^d)$,  $\G_L$ and $\H_L$ are bounded on $L^p(\R^d)$ for all $p \in (1, 2]$. 
\end{teo}

\begin{proof}
By the subordination formula
$$ e^{-t\sqrt{L}} = \frac{1}{\sqrt{\pi}} \int_0^\infty e^{- \frac{t^2}{4s} L} e^{-s} s^{-1/2} ds$$
it follows easily that there exists a positive constant $C$ such that
\begin{equation}\label{11}
 \G_L(f)(x)  \le C \H_L f(x)
 \end{equation}
 for all $f \in L^1(\R^d) \cap L^\infty(\R^d)$ and a.e. $x \in \R^d$. See e.g. \cite{Coulhon-Studia}. Therefore it is enough to prove boundedness of $\H_L$ on $L^p(\R^d)$. \\
 In order to do so, we may consider only nonnegative functions $f \in L^p(\R^d)$. Indeed, for a general $f$ we write 
 $f = f^+ - f^-$ and since
 $$ | \nabla e^{-tL} (f^+-f^-) |^2   \le 2 ( | \nabla e^{-tL} f^+ |^2 + | \nabla e^{-tL} f^- |^2)$$
and 
$$ | \sqrt{V} e^{-tL} (f^+-f^-) |^2   \le 2 ( | \sqrt{V} e^{-tL} f^+ |^2 + | \sqrt{V} e^{-tL} f^- |^2)$$ 
we see that it is enough to prove 
$$\| \H_L(f^+) \|_p + \| \H_L(f^-) \|_p \le C_p(\|  f^+\|_p + \| f^- \|_p),$$
which in turn will imply $\| \H_L(f) \|_p \le 2 C_p \|f\|_p$. 

Now we follow similar arguments as in \cite{Coulhon-Studia}. Fix  a non-trivial $0 \le f \in L^1(\R^d) \cap L^\infty(\R^d)$ and set 
$u(t,x) = e^{-tL}f(x)$. Note that the semigroup $(e^{-tL})_{t \ge0}$ is irreducible (see \cite{Ouh05}, Chapter 4) which means that 
for each $t > 0$, $ u(t,x) > 0$ (a.e.). Observe that 
$$(\frac{\partial}{\partial t} + L) u^p = (1-p) Vu^p - p(p-1) u^{p-2} | \nabla u |^2.$$
This implies
\begin{equation}\label{12}
p | \nabla u|^2 + V |u|^2 =  - \frac{u^{2-p}}{p-1} (\frac{\partial}{\partial t} + L) u^p.
\end{equation}
Hence, there exists a positive constant $c_p$ such that 
\begin{eqnarray*}
\left(\H_L(f)(x) \right)^2 &\le&  - c_p \int_0^\infty u(t,x)^{2-p} (\frac{\partial}{\partial t} + L) u(t,x)^p dt\\
&\le& c_p \sup_{t > 0} u(t,x)^{2-p} J(x)
\end{eqnarray*}
where 
$$J(x) = - \int_0^\infty (\frac{\partial}{\partial t} + L) u(t,x)^p dt.$$
The previous estimate uses the fact that $(\frac{\partial}{\partial t} + L) u(t,x)^p \le 0$ which follows from \eqref{12}. Since the semigroup 
$(e^{-tL})_{t\ge0}$ is sub-Markovian it follows that 
\begin{equation}\label{13}
\| \sup_{t > 0} e^{-tL} f(x) \|_p \le C \| f \|_p.
\end{equation}
The latter estimate is true for all $p \in (1, \infty)$, see \cite{Stein} (p. 73). Therefore, by H\"older's inequality
\begin{equation}\label{14}
\int_{\R^d} | \H_L(f)(x) |^p dx \le c_p \| f \|_p^{\frac{p}{2}(2-p)} \left(\int_{\R^d} J(x) dx \right)^{p/2}.
\end{equation}
On the other hand,
\begin{eqnarray*}
\int_{\R^d} J(x) dx &=& - \int_{\R^d} \int_0^\infty (\frac{\partial}{\partial t} + L) u(t,x)^p dt dx\\
&=& \| f \|_p^p - \int_0^\infty \int_{\R^d} L u(t,x)^p dx dt\\
&=& \| f \|_p^p - \int_0^\infty \int_{\R^d} V u(t,x)^p dx dt\\
&\le&  \| f \|_p^p.
\end{eqnarray*}
Inserting this in \eqref{14} gives
$$\int_{\R^d} | \H_L(f)(x) |^p dx \le c_p \| f \|_p^p$$
which proves the theorem since this estimates extends by density to all $ f \in L^p(\R^d)$. 
\end{proof}

\section{Boundedness on $L^p, \ p > 2$} \label{sec3}
We assume throughout this section that $d \ge 3$. We start with the following result.
\begin{prop}\label{prop31}
Let  $0 \le V \in L_{loc}^1(\R^d)$. If  $\G_L$ (or $\H_L$) is bounded on $L^p(\R^d)$ then there exists a constant $ C > 0$ such that 
\begin{equation}\label{31}
\| \nabla e^{-tL} f \|_p \le \frac{C}{\sqrt{t}} \| f \|_p 
\end{equation}
for all $ t > 0$ and all $f \in L^p(\R^d)$. 
\end{prop}
\begin{proof}
Remember that by \eqref{11}, if $\H_L$ is bounded on $L^p(\R^d)$ then the same holds for $\G_L$. \\
Suppose that $\G_L$ is bounded on $L^p(\R^d)$. We prove that 
\begin{equation}\label{32}
\| \nabla f \|_p \le C \left[ \| L^{1/2} f \|_p + \| L f \|_p^{1/2} \| f \|_p^{1/2} \right].
\end{equation}
The inequality here holds for $f$ in the domain of $L$, seen as an operator on $L^p(\R^d)$.\footnote{Since the semigroup $e^{-tL}$ is sub-Markovian, it acts on $L^p(\R^d)$ and hence the generator of this semigroup in $L^p(\R^d)$ is well defined. This is the operator $L$ we consider on $L^p(\R^d)$.}
In order to do this we follow some arguments from \cite{CoulhonCPAM}. Set $P_t := e^{-t \sqrt{L}}$ and fix $f \in L^2(\R^d)$. 
By integration by parts,
$$\| \nabla P_t f \|_2^2 = (-\Delta P_tf, P_t f) \le (L P_t f, P_tf) = \| L^{1/2} P_tf \|_2^2.$$
In particular, 
$$\| \nabla P_t f \|_2 \le \frac{C}{t} \|f\|_2 \to 0 \ \text{as} \ t \to + \infty.$$
The same arguments show that $t \| \nabla L^{1/2} P_t f \|_2 \to 0$ as $t \to +\infty$. Therefore,
\begin{eqnarray*}
| \nabla f  |^2 &=& - \int_0^\infty \frac{d}{dt} | \nabla P_t f |^2 dt\\
&=& - \left[ t \frac{d}{dt} | \nabla P_t f |^2 \right]_0^\infty  + \int_0^\infty \frac{d^2}{dt^2} | \nabla P_t f |^2 t\, dt \\
&\le& \int_0^\infty \frac{d^2}{dt^2} | \nabla P_t f |^2 t\, dt \\
&=& 2 \int_0^\infty ( |\nabla L^{1/2} P_t f|^2 + \nabla L P_t f.\nabla P_tf) t\, dt\\
&=:& I_1 + I_2.
\end{eqnarray*}
Using the fact that $\G_L$ is bounded on $L^p(\R^d)$ it follows that 
\begin{equation}\label{33}
 \| I_1 \|_{p/2} \le \| \G_L(L^{1/2} f) \|_p^2  \le C  \| L^{1/2} f \|_p^2.
 \end{equation}
 By the Cauchy-Schwartz inequality, 
 \begin{eqnarray*}
 | I_2| &\le& \left( \int_0^\infty ( |\nabla L P_t f|^2 t dt \right)^{1/2} \left( \int_0^\infty ( |\nabla  P_t f|^2 t dt \right)^{1/2} \\
 &\le& \G_L (L f) \G_L(f).
 \end{eqnarray*}
 Integrating gives
 \begin{equation}\label{34}
 \| I_2 \|_{p/2}^{p/2} \le \left( \int_{\R^d} | \G_L(Lf) |^p \right)^{1/2} \left( \int_{\R^d} | \G_L(f) |^p \right)^{1/2} \le C \| L f \|_p^{p/2} \| f \|_p^{p/2}.
 \end{equation}
 Combining \eqref{33} and \eqref{34} gives \eqref{32} for $f \in D(L) \cap L^2(\R^d)$. In order to obtain \eqref{32} for all
 $f \in D(L)$ we take a sequence $f_n \in L^2(\R^d) \cap L^p(\R^d)$ which converges in the $L^p$-norm to $f$. We apply \eqref{32} to 
 $e^{-tL}f_n$ (for $t > 0$) and then  let $n \to + \infty$ and $t \to 0$. 
 
 For $f \in L^p(\R^d)$ we apply \eqref{32} to $e^{-tL}f$ and we note that
 $\| L^{1/2} e^{-tL} f \|_p \le \frac{C}{\sqrt{t}} \| f \|_p $ and $\| L e^{-tL} f \|_p \le \frac{C}{t} \| f \|_p $. Both assertions here follow from the analyticity of the semigroup on $L^p(\R^d)$ (see \cite{Ouh05}, Chap. 7). This proves the proposition.
\end{proof}

\noindent {\bf Remark.} In the proof we did not use the boundedness of the function $\G_L$ but only its gradient part, i.e. boundedness on $L^p(\R^d)$ of the Littlewood-Paley-Stein function:
\begin{equation}\label{G}
 \G(f)(x) = \left( \int_0^\infty t | \nabla e^{-t \sqrt{L}} f(x) |^2 dt \right)^{1/2}.
 \end{equation}

In the next result we shall need the assumption that there exists  $\varphi \in L^\infty(\R^d)$, $\varphi > 0$ such that 
\begin{equation}\label{35}
L \varphi = 0.
\end{equation}
The meaning of \eqref{35} is $e^{-tL} \varphi = \varphi$ for all $t \ge 0$.  \\
Note that \eqref{35} is satisfied for a wide class of potentials. This is the case for example if $V \in L^{d/2-\eps}(\R^d) \cap L^{d/2+\eps}(\R^d)$ for some $\eps > 0$, see \cite{McGillivray-Ouhabaz}. See also \cite{Mu} for more results in this direction. 
\begin{teo}\label{thm32}
Suppose that there exists  $0 < \varphi \in L^\infty(\R^d)$  which satisfies \eqref{35}. Then $\G_L$ (or $\H_L$) is bounded on $L^p(\R^d)$ for some $p > d$ {\it if and only if} $V = 0$
\end{teo}
\begin{proof} If $V = 0$ then  $L = -\Delta$ and it is known that the Littlewood-Paley-Stein function $\G_L$ (and also $\H_L$) is bounded on $L^p(\R^d)$ for all $p \in (1, \infty)$.\\
Suppose now that $V$ is as in the theorem and $\G_L$ is bounded on $L^p(\R^d)$ for some $p > d$.\\
Let $k_t(x,y)$ be the heat kernel of $L$, i.e.,
$$e^{-tL}f(x) = \int_{\R^d} k_t(x,y) f(y) dy$$
for all $f \in L^2(\R^d)$. As mentioned  in the introduction, due to the positivity of $V$, 
\begin{equation}\label{36}
 k_t(x,y) \le \frac{1}{(4 \pi t)^{d/2}} e^{- \frac{|x-y|^2}{4t}}.
 \end{equation}
On the other hand,  using  the Sobolev inequality (for $p > d$)
 $$ |f(x) - f(x') | \le C |x-x'|^{1- \frac{d}{p}} \| \nabla f \|_p$$
  we have
 $$| k_t(x,y) - k_t(x',y) | \le C |x-x'|^{1- \frac{d}{p}} \| \nabla k_t(.,y)\|_p.$$
 Using \eqref{36}, Proposition \ref{prop31} and the fact that 
 $$ k_t(x,y) = e^{- \frac{t}{2}L}k_{\frac{t}{2}}(., y) (x), $$
 we have
 \begin{equation}\label{37}
 | k_t(x,y) - k_t(x',y) |  \le C |x-x'|^{1- \frac{d}{p}}  t^{-\frac{1}{2}} t^{-\frac{d}{2}(1- \frac{1}{p})}.
 \end{equation}
 Thus, using again \eqref{36} we obtain 
 \begin{eqnarray*}
 | k_t(x,y) - k_t(x',y) | &=& | k_t(x,y) - k_t(x',y) |^{1/2} | k_t(x,y) - k_t(x',y) |^{1/2}\\
 &\le& C | x-x'|^{\frac{1}{2}-\frac{d}{2p}}  t^{-\frac{d}{2} +  \frac{d}{4p}- \frac{1}{4}} \left( e^{- \frac{ |x-y|^2 }{8t}} + e^{- \frac{|x'-y|^2}{8t}} \right).
 \end{eqnarray*}
Hence, for $x, x' \in \R^d$
 \begin{eqnarray*}
 | \varphi(x) - \varphi(x') | &= & | e^{-tL} \varphi (x) - e^{-tL}\varphi(x') |\\
 &=& | \int_{\R^d} [ k_t(x,y) - k_t(x',y)] \varphi(y) dy\\
 &\le& \| \varphi \|_\infty \int_{\R^d} | k_t(x,y) - k_t(x',y) | dy\\
 &\le& C | x-x'|^{\frac{1}{2}-\frac{d}{2p}}  t^{\frac{d}{4p}- \frac{1}{4}}.
 \end{eqnarray*}
Letting $t \to \infty$, the RHS converges to $0$ since $p > d$. This implies that $\varphi = c > 0$ is constant. The equality
$0 = L \varphi = L c = Vc$ and hence $V = 0$. 
\end{proof}

\noindent{\bf Remark.} 1. The above proof is inspired from \cite{Magniez} in which it is proved that the boundedness of the Riesz transform $\nabla L^{-1/2}$ on $L^p(\R^d)$  for some $p > d$ implies that $V = 0$. \\
2. According to a previous remark, we could replace in the last theorem the boundedness of $\G_L$ by the boundedness of $\G$ defined by
\eqref{G}.

\end{document}